\documentclass[reqno, a4paper]{amsart}
\usepackage{amssymb, amscd, amsfonts,  amsthm}

\usepackage[mathscr]{eucal}
\usepackage{graphicx}
\usepackage[all,poly]{xy} 
 \usepackage[all]{xy}

\newtheorem{theorem}{Theorem}[section]

\newtheorem{corollary}[theorem]{Corollary}
\newtheorem{proposition}[theorem]{Proposition}

\newtheorem{fact}{Fact}

\theoremstyle{definition}
\newtheorem{definition}[theorem]{Definition}

\newcommand{\restrict}{\,{\mathbin{\vert\mkern-0.3mu\grave{}}}\,}

\newcommand{\remove}[1]{}

\DeclareMathOperator{\cone}{\rm cone}
\DeclareMathOperator{\Rn}{{\mathbb R^{\it n}}}

\DeclareMathOperator{\McNn}{\mathcal M([0,1]^{\it n})}

\DeclareMathOperator{\McN}{\mathcal M}
\DeclareMathOperator{\conv}{\rm conv}

\DeclareMathOperator{\I}{[0,1]}
\DeclareMathOperator{\cube}{[0,1]^{\it n}}

\DeclareMathOperator{\interior}{\rm int}

 \title[Strong semisimplicity and tangents]
{Bouligand-Severi tangents
in MV-algebras}

\author{Manuela Busaniche and Daniele Mundici}

\address[M.Busaniche]{Instituto de Matem{\'a}tica Aplicada del Litoral \\
CONICET-UNL \\
G\"{u}emes 3450\\ S3000GLN-Santa Fe\\
Argentina}

\email{ mbusaniche@santafe-conicet.gov.ar }

\address[D. Mundici]{Dipartimento di
Matematica ``Ulisse Dini'' \\
Universit\`{a} degli Studi di Firenze \\
viale Morgagni 67/A \\
I-50134 Firenze \\
Italy}

\email{ mundici@math.unifi.it }

\date{\today}

\begin{document}


   \begin{abstract}
In their recent seminal paper
published in the Annals of Pure and Applied
Logic,   Dubuc and Poveda call an
MV-algebra $A$  {\it strongly semisimple} if all  principal
quotients of $A$  are semisimple.
All boolean algebras are strongly semisimple, and so are
all finitely presented MV-algebras.
We  show that for  any 1-generator
MV-algebra semisimplicity is equivalent
to strong semisimplicity.
Further,  a semisimple 2-generator
MV-algebra  $A$ is strongly semisimple iff
its maximal spectral  space
$\mu(A)\subseteq \I^2$  does not have any
rational Bouligand-Severi tangents at its rational points. 
In general, when  $A$ is  finitely generated
  and $\mu(A)\subseteq \cube$  has a Bouligand-Severi
tangent then $A$  is not strongly semisimple.
\end{abstract}

\maketitle

\section{Introduction: stable consequence}
We refer to \cite{cigdotmun} and \cite{mun11} for background on
MV-algebras.
Following Dubuc and Poveda \cite{dubpov},
we say that an MV-algebra $A$  is
{\it strongly semisimple}  if for every principal ideal
$I $  of $A$ the quotient $A/I$  is semisimple.
 Since  $\{0\}$ is
a principal ideal of $A$,  every
strongly semisimple MV-algebra is
semisimple.

From a classical result by Hay  \cite{hay} and W\' ojcicki \cite{woj}
(also see \cite[4.6.7]{cigdotmun} and \cite[1.6]{mun11}),  it follows that every
 finitely presented MV-algebra   is  strongly semisimple.
 Trivially, all  hyperarchimedean MV-algebras, whence in particular
 all boolean algebras, are strongly semisimple, and so are
 all simple and all finite MV-algebras, \cite[3.5 and 3.6.5]{cigdotmun}.

Our paper is devoted to characterizing $n$-generator
strongly semisimple MV-algebras for $n=1,2.$
In Theorem \ref{theorem:onedim} we show that when  $n=1$
 strong semisimplicity is equivalent to semisimplicity.

 As the reader will recall (\cite[9.1.5]{cigdotmun}),
  the free $n$-generator MV-algebra is the
  MV-algebra  $\McNn$ of
  all McNaughton functions $f\colon \cube\to \I$,
  with pointwise operations of negation $\neg x=1-x$
  and truncated addition $x\oplus y=\min(1,x+y).$

  For any
  nonempty closed set $X\subseteq [0,1]^n$
 we let   $\mathcal M(X)$  denote the MV-algebra of
 restrictions to $X$ of the functions in $\McNn$, in symbols,
 $$\mathcal M(X)=\{f\restrict X\mid f\in \McNn\}.$$
 By \cite[3.6.7]{cigdotmun}, $\mathcal M(X)$ is a semisimple
 MV-algebra---actually, up to isomorphism, $\mathcal M(X)$
 is the most general possible $n$-generator semisimple
 MV-algebra $A$: to see this,
   pick  generators $\{a_1, \ldots , a_n\}$ of   $A$.
Let $	\pi_i\colon \I^n\to \I$  be the projection functions
in the free MV-algebra  $\McN(\I^n)$ for $i=1, \ldots, n.$
Then the assignment that maps $\pi_i\mapsto a_i$ for each $i=1, \ldots, n$  uniquely extends
 to a homomorphism $\eta_a\colon \mathcal M(\I^n)\to A$
 of the free $n$-generator MV-algebra onto $A$.
 Let $\mathfrak h_a=\ker(\eta_a)$ be the kernel of this
 homomorphism and
 \begin{equation}
 \label{equation:zeroset-from-generators}
 \mathcal {Z}_a=\bigcap\{Zf\mid f\in \mathfrak h_a\}
 \end{equation}
 the intersection of the zerosets of the McNaughton functions
 in $\mathfrak h_a$. Then 
  \begin{equation}
 \label{equation:semisimple}
 \mbox{$A\cong \mathcal M(\mathcal Z_a)$.}
 \end{equation}

In Theorem \ref{theorem:notss-implies-tangent} we prove that
a  2-generator  MV-algebra
   $A=\mathcal M(X)$  with   $X\subseteq \I^2$ is strongly
  semisimple   iff $X$ has no rational outgoing Bouligand-Severi tangent vector
at any of its rational points, \cite{bou, sev, rif}.
Having such a tangent is a sufficient condition
for $\mathcal M(X)$ not to be strongly semisimple,
for any $X\subseteq \cube,$  (Theorem \ref{theorem:rational-tangent-implies-notss}).
 Here, as usual, a point
 $x\in \mathbb R^n$  is said to be
 {\it rational}  if so are all its coordinates.

\smallskip
By a {\it rational vector}   we mean
a nonzero vector $w\in \Rn$ such that the line
$\mathbb Rw\subseteq \Rn$ contains at least two rational points.

\smallskip
\noindent
{\it Notation}:  Given $g\in \McN(\I^n)$ let $Zg=\{x\in [0,1]^n \mid g(x)=0\}.$   Following
 \cite[p.33]{cigdotmun} or
\cite[p.21]{mun11}, for $k\in \mathbb{N}$,
$k\centerdot g$
stands for  $k$-fold pointwise truncated addition of $g$.

\section{One-generator MV-algebras}

  \begin{theorem}
  \label{theorem:onedim}
Every  one-generator semisimple
MV-algebra $A$ is strongly semisimple.
\end{theorem}

\begin{proof}
As in (\ref{equation:zeroset-from-generators})-(\ref{equation:semisimple}), let
 $X\subseteq [0,1]$ be a nonempty closed set such that $A\cong \McN(X).$
For some  $g\in \McN(\I)$ let $J$ be the principal ideal of
$\McN(\I)$ generated by $g$,  and
$J^{\shortmid}$ be the principal
ideal of  $\mathcal M(X)$ generated by  $g^{\shortmid}=g\restrict X.$
Observe that  $J^\shortmid=\{l\restrict X\mid l \in J\}$.
For every  $f\in \McN(\I)$, letting
$f^\shortmid = f\restrict X$   we must prove:  if
$f^\shortmid$  belongs to all maximal ideals of
$\McN(X)$ to which $g^\shortmid $  belongs, then
$f^\shortmid$ belongs to $J^\shortmid$.
In the light of  \cite[3.6.6]{cigdotmun} and \cite[4.19]{mun11},
this amounts to proving
\begin{equation}
\label{equation:ebbene}
\mbox{if  }
f =0 \mbox{ on } Zg\cap X
\mbox{  then   } f\restrict X\in J^{\shortmid}.
\end{equation}
Let $\Delta$ be a triangulation of $\I$  such that $f$ and $g$ are
linear over every simplex of $\Delta.$
The existence of $\Delta$  follows from
the piecewise linearity of $f$ and $g$, \cite{sta}.
In view of the compactness of $X$ and $\I$,
it  is sufficient to settle the following

\medskip
\noindent {\it Claim.}
Suppose  $ f\in \mathcal M(\I)$
vanishes over $Zg\cap X$. Then for
all $x\in X$ there is an open neighbourhood
$\mathcal N_x\ni x$ in $\I$
 together with an integer $m_x\geq 0$ such that  $m_x\centerdot
 g\geq f$   on $\mathcal N_x\cap X.$

\medskip
We proceed by cases:

\medskip
\noindent
{\it Case 1:}  $g(x)>0.$  Then for some  integer  $r$ and open neighbourhood
$\mathcal N_x\ni x$ we have  $g>1/r$  over $\mathcal N_x$. Letting
 $m_x= r$ we have  $1=m_x\centerdot g\geq f$
over $\mathcal N_x,$ whence a fortiori,
$m_x\centerdot g\geq f$ over  $\mathcal N_x\cap X.$

\medskip
\noindent
{\it Case 2:}  $g(x)=0.$  Since $ f$
vanishes over $Zg\cap X$,  then
$f(x)=0.$
Let  $T$ be a 1-simplex of $\Delta$  such that $x\in T.$
Let $T_x$ be the smallest face of $T$ containing $x$.

\medskip
\noindent
{\it Subcase 2.1:}  $T_x=T$.  Then $x\in \interior(T)$. Since $g$ is linear over $T$
then $g$ vanishes over $T$. By our hypotheses on $f$ and $\Delta$,
 $f$ vanishes over $T$, whence
and $0=g\geq f=0$ on $T$. Letting $\mathcal N_x=\interior(T)$ and
$m_x=1$, we get $m_x\centerdot g\geq f$  over $\mathcal N_x$
whence a fortiori,
the inequality holds over  $\mathcal N_x\cap X.$

\medskip
\noindent
{\it Subcase 2.2:} $T_x=\{x\}$. Then   $T=\conv(x,y)$ for some $y\not=x$.
Without loss of generality,  $y>x.$ We will exhibit  a
right open neighbourhood  $\mathcal R_x \ni x$ and an integer
$r_x\geq 0$ such that
 $r_x\centerdot
 g\geq f$   on $\mathcal R_x\cap X.$  The same argument yields
 a left neighbourhood  $\mathcal L_x\ni x$  and
 an integer
$l_x\geq 0$ such that
 $l_x\centerdot
 g\geq f$   on $\mathcal L_x\cap X.$  One then takes
 $\mathcal N_x=\mathcal R_x \cup \mathcal L_x $
 and $m_x=\max(r_x,l_x).$

\medskip
\noindent
{\it Subsubcase 2.2.1:}  If both $g$ and $f$ vanish at $y$, then they vanish
over $T$  (because they are linear over $T$). Upon defining
$\mathcal R_x=\interior(T)\cup\{x\}$ and
$r_x=1$ we get $r_x\centerdot g\geq f$ over $\mathcal R_x$,
whence in particular, over $\mathcal R_x\cap X.$

\medskip
\noindent
{\it Subsubcase 2.2.2:} If both $g$ and $f$ are $>0$ at $y$  then for all
suitably large $m$  we have $1=m\centerdot g\geq f$ on $T$.
Letting  $r_x$ the smallest such $m$ and
$\mathcal R_x=\interior(T)\cup\{x\}$ we have the desired inequality
over  $\mathcal R_x$ and a fortiori over $\mathcal R_x\cap X$.

\medskip
\noindent
{\it Subsubcase 2.2.3:}  $g(y)=0, f(y)>0$.  By our hypotheses on
 $\Delta$,  $g$ is linear over $T$
 and hence $g=0$  over $T$. It follows that
$X\cap T = \{x\}$:   for otherwise,  
our assumption $Zf\cap X\supseteq Zg\cap X$
together with the linearity of $f$ over $T$
would imply  $f(y)=0,$  against our current hypothesis.
Letting   $\mathcal R_x=\interior(T)\cup\{x\}$
and  $r_x=1$ we have $r_x\centerdot g\geq f$  over
$\mathcal R_x\cap X$.
\end{proof}

 \section{Strong semisimplicity and Bouligand-Severi tangents}
 Severi \cite[\S 53, p.59 and p.392]{sev0},
 \cite[\S1, p.99]{sev} and independently, Bouligand
 \cite[p.32]{bou} called a half-line $H\subseteq \mathbb R^n$  {\it tangent} to a
 set $X\subseteq \mathbb R^n$ at an accumulation point $x$
 of $X$ if for all $\epsilon, \delta>0$ there is $y\in X$ other than $x$
such that  $||y-x||<\epsilon$, and the angle between $H$ and
 the half-line through $y$ originating at $x$ is $<\delta$.

 Here as usual,  $||v||$ is the length of vector  $v\in \Rn$.

     Severi \cite[\S 2, p. 100 and \S 4, p.102]{sev} noted that
     for any accumulation point $x$ of a closed set $X$ there
     is  a half-line $H$  tangent to $X$ at $x$.

Today  {(see, e.g., \cite[p.16]{book}, \cite[p.1376]{rif})},
Bouligand-Severi tangents are routinely  introduced as
follows:

\begin{definition}
\label{definition:B-S}
Let  $x$ be an element of   a  closed subset $X$ of $\Rn,$
and $u$  a unit vector in $\mathbb R^n.$
We then say that $u$ is a  {\it
 Bouligand-Severi  tangent  (unit) vector to $X$    at   $x$}
 if $X$ contains  a sequence $x_0,x_1,\dots$ of elements,
 all different from $x$,  such that
$$
\lim_{i\to \infty }x_i =  x\,\,\, \mbox{ and } \,\,\,\lim_{i\to \infty } {(x_i-x)}/{||x_i-x||}= u.
$$
 \end{definition}

 Observe that $x$ is an accumulation point of $X$.
   We further say that $u$  is   {\it outgoing }  if for some
 $\lambda >0$  the segment $\conv(x,x+\lambda u)$
 intersects $X$ only at $x.$

Already Severi
noted that  his definition of tangent half-line
$H=\mathbb R_{\geq 0}u$
is equivalent to
Definition \ref{definition:B-S}:

  \begin{proposition}
 \label{proposition:reformulation}
 {\rm  (\cite[\S 5, p.103]{sev}).}
 For any  nonempty closed subset $X$ of $\Rn$,
point  $x\in X$,  and   unit vector  $u\in \Rn$
 the following conditions are equivalent:
 \begin{itemize}
\item[(i)]  For all $\epsilon, \delta>0$,  the cone
  $\cone_{x,u,\epsilon,\delta}$
with  apex $x$, axis
 parallel to $u$, vertex angle $2\delta$  and height $\epsilon$
contains infinitely many points of $X$.

\smallskip
 \item[(ii)]   $u$ is a Bouligand-Severi tangent vector to $X$ at $x$.
 \end{itemize}
 \end{proposition}

 When $n=1,$ $\cone_{x,u,\epsilon,\delta}$ is the
 segment $\conv(x,x+\epsilon u)$.
 When  $n=2,$
 $\cone_{x,u,\epsilon,\delta}$ is the
 isosceles triangle  $\conv(x,a,b)$ with vertex $x$,
 basis $\conv(a,b),$  height equal to $\epsilon$  and
 vertex angle  $\widehat{axb}=2\delta.$


\bigskip
\noindent
 The next two results provide  geometric
 necessary and sufficient conditions
 on $X$ for  the semisimple MV-algebra
  $\mathcal M(X)$
 to be strongly semisimple.
 These conditions are stated in terms of
 the non-existence of Bouligand-Severi tangent vectors
 having certain rationality properties.


 \begin{theorem}
\label{theorem:rational-tangent-implies-notss}
Let $X$ be a nonempty closed set  in $[0,1]^n$.
Suppose  $X$ has
 a Bouligand-Severi   rational outgoing  tangent vector $u$
at some rational point $x\in X$. Then
  $\mathcal M(X)$ is not strongly semisimple.
\end{theorem}

\begin{proof}
Since $u$ is outgoing, let $\lambda>0$ satisfy
$X\cap \conv(x,x+\lambda u)=\{x\}$.
Without loss of generality  $x+\lambda u \in \mathbb Q^n$.
Our hypothesis together with { Proposition  \ref{proposition:reformulation} }
yields a sequence  $w_1, w_2,\dots$   of distinct
points of $X$, all distinct from $x$,
accumulating at $x$, at strictly decreasing
distances from $x$,  in such a way that
the sequence of unit vectors  $u_i$ given by
$
({w_i-x})/{||w_i-x||}
$
tends to $u$ as $i$ tends to $\infty.$
Let  $y=x+\lambda u.$
Since $X\cap \conv(x,y)=\{x\}$,
  no point $w_i$ lies on the segment
 $\conv(x,y)$, and we can further assume that the sequence of angles
 $\widehat{w_ixy}$ is strictly decreasing and tends to zero
 as $i$ tends to $\infty$.

Since both points $x$  and $y$ are  rational, then 
by  \cite[2.10]{mun11} for some
$g\in \mathcal M(\cube)$  the zeroset
$$Zg=\{z\in \cube\mid g(z)=0\}$$
coincides with the segment
$\conv(x,y)$.  Thus,
   $$\frac{\partial g(x)}{\partial(u)}=0.$$
Let $J$ be the ideal of $\mathcal M(\cube)$
  generated by $g$,
  $$
  J=\{f\in\McNn\mid f\leq k\centerdot g \mbox{ for some $k=0,1,2,\ldots$}\}.
  $$
%
Then for each $f\in J$,
   $$\frac{\partial f(x)}{\partial(u)}=0.$$
Since the directional derivatives of $f$ at $x$ are continuous,
(meaning that the map $t\mapsto \partial f(x)/\partial t$ is continuous)
 it follows that
\begin{equation}
\label{equation:criterion}
   \lim_{t\to u}\frac{\partial f(x)}{\partial(t)}=\frac{\partial f(x)}{\partial(u)}=0.
\end{equation}
 Let $g^\shortmid =g\restrict X$  and
$$J^\shortmid=
\{f^\shortmid\in \mathcal M(X)\mid
 \, f^\shortmid\leq k\centerdot g^\shortmid
 \mbox{ for some $k=0,1,2, \ldots$}\}
 $$
 be the ideal of $\mathcal M(X)$ generated by $g^\shortmid$.
A moment's reflection shows that
\begin{equation}
\label{equation:ideals}
J^\shortmid=\{l \restrict X\mid l \in J\}.
\end{equation}
One inclusion is trivial.
For the converse inclusion, if
$f\restrict X \leq (k\centerdot g)\restrict X$ then
letting  $l=f\wedge k\centerdot g$ we get
$l \leq k\centerdot g$.
So $l\in J$  and $l\restrict X= f\restrict X,$ showing that
$f\restrict X $ is extendible to some $l\in J.$

For any  $f \in \McNn,$
the  piecewise linearity of $f$
ensures that  for all large $i$ the value of the incremental ratio
$({f(w_i)-f(x)})/{||w_i-x||}$
coincides with the directional derivative
${\partial f(x)}/{\partial u_i}$
along the unit vector $u_i=(w_i-x)/||w_i-x||$.
Thus in particular,  if  $f\restrict X = f^\shortmid\in J^\shortmid$, from
 (\ref{equation:criterion})-(\ref{equation:ideals}) it follows that
$$
\lim_{i\to\infty}\frac{f^\shortmid(w_i)-f^\shortmid(x)}{||w_i-x||}= 0.
$$

Since $x$ is rational, again by \cite[2.10]{mun11} there is
  $j\in \McNn$  with $Zj=\{x\}$.
  For some  $\omega>0$ we have
$\partial j(x)/\partial(u)=\omega,$ whence
$$
\lim_{i\to\infty}\frac{j^\shortmid(w_i)-j^\shortmid(x)}{||w_i-x||} = \omega.
$$
Therefore,   $j^\shortmid\notin J^\shortmid$.
Since $Zg\cap X=\{x\}$,
recalling  \cite[4.19]{mun11}
we see that
the only maximal of $\mathcal M(X)$
containing  $J^\shortmid$  is the set
of all functions in   $\mathcal M(X)$ that
vanish at $x$.
Thus,    $j^\shortmid$ belongs to all maximal ideals
of $\mathcal M(X)$  containing  $J^\shortmid$.
   By  \cite[3.6.6]{cigdotmun},
  $ \mathcal  M(X)$ is not strongly semisimple: specifically,  $j'/J'$ is
  infinitesimal in the principal quotient   $ \mathcal  M(X)/J'$.
\end{proof}

  As a partial converse we have:


 \begin{theorem}
\label{theorem:notss-implies-tangent}
Let $X\subseteq [0,1]^n$ be a nonempty closed set.
Suppose the MV-algebra
$\mathcal M(X)$ is not strongly semisimple.

\begin{itemize}
\item[(i)]  Then $X$ has a Bouligand-Severi  tangent
 vector  $u$ at some point $x\in X$ satisfying the following
 {\rm nonalignment}  condition: there is a sequence of distinct  $w_i\in X$,
 all distinct from $x$ such that
 $$\lim_{i\to \infty}w_i=x,\,\,\,\,\,
 \lim_{i\to \infty}\frac{w_i-x}{||w_i-x||}=u, \,\,\,\,\,
 w_i\notin \conv(x,x+u)\,\, \mbox{for all }\,i.
 $$

\medskip
\item[(ii)] In particular, if $n=2$, then
$X$ has a Bouligand-Severi
  {\rm outgoing rational}
 tangent
 vector $u$
at some {\rm rational}  point $x\in X.$
\end{itemize}
\end{theorem}

\begin{proof}
(i) 
The hypothesis yields a  function
$g\in \McNn$,  with its restriction
  $g^\shortmid= g\restrict X \in \mathcal M(X),$
  in such a way that
  the principal ideal
  $J^\shortmid$ of $\mathcal M(X)$
  generated by $g^\shortmid,$
$$
J^\shortmid=
\{l^\shortmid\in \mathcal M(X)  \mid l^\shortmid
 \leq k\centerdot g^\shortmid  \,\, \mbox{for some } k=1,2,\dots\}
$$
is  strictly contained in the
intersection  $I$ of all maximal ideals of $\mathcal M(X)$  containing $J'$.
Thus    for some $j\in  \McNn$ letting $j^\shortmid=j\restrict X$
we have  $j^\shortmid\in I\setminus J^\shortmid.$
By     \cite[3.6.6]{cigdotmun} and \cite[4.19]{mun11},
\begin{equation}
\label{equation:zeroforzero}
j^\shortmid=0 \mbox{ on } Zg^\shortmid,\,\,\, {\rm i.e.,}
 \,\,\,X\cap Zj \supseteq X\cap Zg
\end{equation}
and
\begin{equation}
\label{equation:preliminary-nodominance}
 \forall m=0,1,\dots\exists z_m\in X,\,\,j^\shortmid(z_m)
 > m\centerdot g^\shortmid(z_m).
\end{equation}
There is a sequence  of integers $0< m_0<m_1<\dots$ and a subsequence
$y_0,y_1,\dots$  of   $\{z_i,z_2,\dots\}$  such that
$y_i\not= y_l$  for $i\not=l$   and

\begin{equation}
\label{equation:infinitesimal-dominates}
 \forall t =0,1,\dots,\,\,\,  j^\shortmid(y_t)>m_t\centerdot g^\shortmid(y_t).
\end{equation}
The compactness of  $X$ yields
  an accumulation point $x\in X$
of  the $y_t$.
Without loss of generality  (taking a subsequence, if
necessary) we can further assume

\begin{equation}
\label{equation:approaching}
||y_0-x||>||y_1-x||>\cdots,   \mbox{ whence }
 \lim_{i\to\infty}  y_i=x.
\end{equation}
By (\ref{equation:infinitesimal-dominates}), for all $t$,
$j^\shortmid(y_t)>0$. Then by  (\ref{equation:zeroforzero}),
$g^\shortmid(y_t)>0$.
%
%
For each $i=0,1,\dots$, letting  the unit vector $u_i\in \Rn$
be defined by
$
u_i= (y_i-x)/{||y_i-x||},
$
we obtain  a sequence
 of (possibly repeated) unit vectors   $u_i\in \Rn$.
 Since the boundary of the unit ball in $\Rn$ is compact,
some  unit vector
$u\in \Rn$ satisfies
$$
\forall \epsilon>0  \mbox{ there are infinitely many $i$ such that }  ||u_i-u||<\epsilon.
$$
Some subsequence
$w_0,w_1,\dots$
of the  $y_i$ will satisfy the condition
\begin{equation}
\label{equation:iconale-bis}
\forall \epsilon, \delta>0 \mbox{ there is $k$ such that for all $i>k$, }\,\,\,
 w_i\in \cone_{x,u,\epsilon,\delta}.
\end{equation}
Correspondingly,  the sequence  $v_0,v_1,\dots$
 given by $v_k=(w_k-x)/{||w_k-x||}$  will satisfy
 %
 %
 %
 %
\begin{equation}
\label{equation:lim-vi=u}
\lim_{i\to \infty} v_i=u.
\end{equation}
We have just proved that  $u$ is a Bouligand-Severi tangent to $X$ at $x$.

 \medskip
To complete the proof of (i) we prepare:

\medskip
  \begin{fact}
\label{fact:g(x)=0}
$g^\shortmid(x)=0$.
\end{fact}
Otherwise, from the continuity of $g$, for some real
$\rho>0$ and  suitably small $\epsilon >0$,
we have the inequality $g(z)>\rho$ for all
$z$ in
the open ball
$B_{x,\epsilon}$ of radius $\epsilon$  centered at $x$.
By (\ref{equation:iconale-bis}), $B_{x,\epsilon}$ contains
 infinitely many $w_i$. There is a fixed
 integer     $\bar m>0$  such
that $1=\bar{m}\centerdot g^\shortmid \geq j^\shortmid$  for all these
$w_i$, which contradicts (\ref{equation:infinitesimal-dominates}).

\medskip

\begin{fact}
\label{fact:j(x)=0}
$j^\shortmid(x)=0.$
\end{fact}
 This immediately follows
from  (\ref{equation:zeroforzero}) and {  Fact \ref{fact:g(x)=0}.}

\medskip
\begin{fact}
\label{fact:g-u-flat}
$ {\partial g(x)}/{\partial u}=0.$
\end{fact}

By way of contradiction, suppose ${\partial g(x)}/{\partial u}=\theta>0.$
In view of the continuity of
the map $t\mapsto {\partial g(x)}/{\partial t}$,
let  $\delta>0$ be such that
 ${\partial g(x)}/{\partial r}>\theta/2,$
 for any unit vector   $r$
 such that   $\widehat{ru}<\delta.$
 Since  by { Fact \ref{fact:j(x)=0}}  $j(x)=0$
 and both $g$ and $j$ are piecewise linear,
 there is
 an $\epsilon >0$ together with an integer
 $\bar k>0$ such that
$\bar k\centerdot g\geq j$ over the
cone
 $C=\cone_{x,u,\epsilon,\delta}$.
{ By  (\ref{equation:iconale-bis}),}
$C$  contains
infinitely many $w_i,$ {  in contradiction with
 (\ref{equation:infinitesimal-dominates}).}

 \medskip
To conclude the proof of the nonalignment
 condition in   (i),  it is sufficient to
 settle the following:

\begin{fact}
\label{fact:nonalignment}
There is $\lambda>0$  such
that for all large $i$
the segment $\conv(x,x+\lambda u)$ contains no $w_i.$
\end{fact}

For otherwise,  from { Fact \ref{fact:g-u-flat}}
  $\partial g(x)/\partial(u)=0,$ whence
  the piecewise linearity of $g$ ensures that
$g$ vanishes on infinitely many  $w_i$ of
$\conv(x,x+\lambda u)$   arbitrarily near $x$.
Any such $w_i$   belongs to $X$, whence
{ by (\ref{equation:zeroforzero})},
   $j(w_i)=0$,  in contradiction with
   {    (\ref{equation:infinitesimal-dominates})}.

\medskip The proof of (i) is now complete.

\bigskip
(ii)    Let  $H^\pm$  be the two closed half-spaces of $\mathbb R^2$
  determined by the line passing through $x$ and $x+u$.
  By (\ref{equation:iconale-bis}),  infinitely many
  $w_i$ lie in the same closed half-space, say, $H^+$.
Without loss of generality,
 $H^+\cap \interior(\I^2)\not=\emptyset.$
   Let  $u^\perp$ be the orthogonal vector to $u$
such that  $x+u^{\perp}\in H^+$.

\medskip

\begin{fact}
\label{fact:six-hic}  For all small $\epsilon >0,$
$$
\frac{\partial g(x+\epsilon u)}{\partial u^\perp}>0.
$$
\end{fact}
By way of contradiction, assume
$
{\partial g(x+\epsilon u)}/{\partial u^\perp}=0.
$
Since $g$ is piecewise linear, by { Facts  \ref{fact:g(x)=0}  and  \ref{fact:g-u-flat}},
for suitably small $\eta,\omega>0,$ the function   $g$
 vanishes over the triangle
$T=\conv(x, x+\eta u, x+\eta u+\omega u^\perp)$.
By (\ref{equation:iconale-bis}),  $\,\,T$
contains infinitely many  $w_i.$
By  (\ref{equation:zeroforzero}),    $g(w_i)=j(w_i)=0$
 against (\ref{equation:infinitesimal-dominates}).

\medskip
 \begin{fact}
\label{fact:j-jumps}
$$
\frac{\partial j(x)}{\partial u}>0.
$$
\end{fact}
Otherwise,
$
{\partial j(x)}/{\partial u}=0.
$
{Fact \ref{fact:six-hic}}   yields a fixed integer $\bar h$
such  that,  on a suitably small triangle
of the form
$T=\conv(x,x+\epsilon u,x+\epsilon u+\omega u^\perp)$,
we have   $\bar h\centerdot g\geq j$.
By (\ref{equation:iconale-bis}),
 $T$ contains
infinitely many  $w_i,$  again contradicting
(\ref{equation:infinitesimal-dominates}).


\medskip

We now prove a { strong form of
  Fact  \ref{fact:nonalignment}},  showing that
  $u$ is an {\it outgoing} tangent vector:

\medskip

\begin{fact}
\label{fact:outgoing-hic}
For some $\lambda>0$ the segment $\conv(x,x+\lambda u)$
intersects $X$ only at $x$.
\end{fact}

Otherwise,  from    {   Facts \ref{fact:g(x)=0}
and  \ref{fact:g-u-flat}}
it follows that
 $g$ vanishes on infinitely many points of
 $X\cap \conv(x,x+\lambda u)$
 converging to $x$.
 By  (\ref{equation:zeroforzero}),
   $j^\shortmid$ vanishes on all these points.
Since  $j$ is  piecewise linear,   $\partial j(x)/\partial u=0$, against
 {Fact \ref{fact:j-jumps}.}

\bigskip

By a {\it rational}  line in $\Rn$  we mean a line passing through
at least two distinct rational points.

\medskip

 \begin{fact}
\label{fact:eight}  $x$ is a rational point, and $u$ is a rational vector.
\end{fact}
As a matter of fact,
{Facts \ref{fact:j-jumps} and \ref{fact:j(x)=0}}  yield a
  rational line  $L$
through $x$.
On the other hand,   { Facts  \ref{fact:g-u-flat} and \ref{fact:six-hic}}
show that the line   passing through $x$ and $x+u$
is rational and different from $L$.  Thus $x$ is rational, whence so is the
vector $u$.

\bigskip
We conclude
that $X$ has $u$ as a
Bouligand-Severi  {\it outgoing rational } tangent  vector
at the rational point $x$.
\end{proof}

%

\bigskip
Recalling 
{ Theorem \ref{theorem:rational-tangent-implies-notss}}
we now obtain:

 \begin{corollary}
\label{corollary:two-dimensional}
Let $X\subseteq [0,1]^2$ be a nonempty closed set.
 Then  $\mathcal M(X)$ is not strongly semisimple
iff $X$ has a  Bouligand-Severi outgoing
rational  tangent vector
$u$ at some rational point   $x\in X$.
\end{corollary}

\bibliographystyle{plain}

\end{document}